\documentclass{amsart}
\usepackage{amsfonts}
\usepackage{graphicx}
\usepackage{amscd}

\newtheorem{theorem}{Theorem}
\theoremstyle{plain}

\newtheorem{definition}{Definition}

\newtheorem{lemma}{Lemma}

\numberwithin{equation}{section}

\input{tcilatex}

\begin{document}
\title[Cauchy's residue theorem...]{Cauchy's residue theorem for a class of
real valued functions}
\author[Branko Sari\'{c}]{Branko Sari\'{c}}
\address{Mathematical Institute, Serbian Academy of Sciences and Arts, Knez
Mihajlova 35, 11001 Belgrade, Serbia; College of Technical Engineering
Professional Studies, Svetog Save 65, 32 000 \v{C}a\v{c}ak, Serbia}
\email{bsaric@ptt.rs}
\date{June 08, 2009}
\thanks{The author's research is supported by the Ministry of Science,
Technology and Development, Republic of Serbia (Project ON144002)}

\begin{abstract}
Let $\left[ a,b\right] $ be an interval in $\mathbb{R}$ and let $F$ be a
real valued function defined at the endpoints of $[a,b]$ and with a certain
number of discontinuities within $\left[ a,b\right] $. Having assumed $F$ to
be differentiable on a set $\left[ a,b\right] \backslash E$ to the
derivative $f$, where $E$ is a subset of $\left[ a,b\right] $ at whose
points $F$ can take values $\pm \infty $ or not be defined at all, we adopt
the convention that $F$ and $f$ are equal to $0$ at all points of $E$ and
show that $\mathcal{KH-}vt\int_{a}^{b}f=F\left( b\right) -F\left( a\right) $%
, where $\mathcal{KH-}$ $vt$ denotes the total value of the \textit{%
Kurzweil-Henstock} integral. The paper ends with a few examples that
illustrate the theory.
\end{abstract}

\maketitle

\section{Introduction}

Let $\left[ a,b\right] $ be some compact interval in $\mathbb{R}$. It is an
old result that for an ACG$_{\delta }$ function $F:\left[ a,b\right] \mapsto 
\mathbb{R}$ on $\left[ a,b\right] $, which is differentiable almost
everywhere on $\left[ a,b\right] $, its derivative $f$ is integrable (in the 
\textit{Kurzweil-Henstock} sense) on $\left[ a,b\right] $ and $\mathcal{KH-}%
\int_{a}^{b}f=F\left( b\right) -F\left( a\right) $, \cite[Theorem 9.17]{3}.
The aim of this note is to define a new definite integral named the total 
\textit{Kurzweil-Henstock} integral that can be used to extend the above
mentioned result to any real valued function $F$ defined and differentiable
on $\left[ a,b\right] \backslash E$, where $E$ is a certain subset of $\left[
a,b\right] $ at whose points $F$ can take values $\pm \infty $ or not be
defined at all. Unless otherwise stated in what follows, we assume that the
endpoints of $\left[ a,b\right] $ do not belong to $E$. Now, define point
functions $F_{ex}:\left[ a,b\right] \mapsto \mathbb{R}$ and $D_{ex}F:\left[
a,b\right] \mapsto \mathbb{R}$ by extending $F$ and its derivative $f$ from $%
\left[ a,b\right] \backslash E$ to $E$ by $F_{ex}\left( x\right) =0$ and $%
D_{ex}F\left( x\right) =0$ for $x\in E$, so that 
\begin{equation}
F_{ex}\left( x\right) =\left\{ 
\begin{array}{c}
F\left( x\right) \text{, if }x\in \left[ a,b\right] \backslash E \\ 
0\text{, if }x\in E%
\end{array}
\right. \text{ and}  \label{1}
\end{equation}
\begin{equation*}
D_{ex}F\left( x\right) =\left\{ 
\begin{array}{c}
f\left( x\right) \text{, if }x\in \left[ a,b\right] \backslash E \\ 
0\text{, if }x\in E%
\end{array}
\right. \text{.}
\end{equation*}

\section{Preliminaries}

A partition $P\left[ a,b\right] $ of $\left[ a,b\right] \in \mathbb{R}$ is a
finite set (collection) of interval-point pairs $\left\{ \left( \left[
a_{i},b_{i}\right] ,x_{i}\right) \mid i=1,...,\nu \right\} $, such that the
subintervals $\left[ a_{i},b_{i}\right] $ are non-overlapping, $\cup _{i\leq
\nu }\left[ a_{i},b_{i}\right] =\left[ a,b\right] $ and $x_{i}\in \left[
a_{i},b_{i}\right] $. The points $\left\{ x_{i}\right\} _{i\leq \nu }$ are
the tags of $P\left[ a,b\right] $, \cite{2}. It is evident that a given
partition\thinspace of $\left[ a,b\right] $ can be tagged in infinitely many
ways by choosing different points as tags. If $E$ is a subset of $\left[ a,b%
\right] $, then the restriction of $P\left[ a,b\right] $ to $E$ is a finite
collection of $\left( \left[ a_{i},b_{i}\right] ,x_{i}\right) \in P\left[ a,b%
\right] $ such that each $x_{i}\in E$. In symbols, $P\left[ a,b\right]
\left| _{E}\right. =\left\{ \left( \left[ a_{i},b_{i}\right] ,x_{i}\right)
\mid x_{i}\in E,\,i=1,...,\nu _{E}\right\} $. Let $\mathcal{P}\left[ a,b%
\right] $ be the family of all partitions $P\left[ a,b\right] $ of $\left[
a,b\right] $. Given $\delta :\left[ a,b\right] \mapsto \mathbb{R}_{+}$,
named a gauge, a partition $P\left[ a,b\right] \in $ $\mathcal{P}\left[ a,b%
\right] $ is called $\delta $\textit{-fine} if $\left[ a_{i},b_{i}\right]
\subseteq \left( x_{i}-\delta \left( x_{i}\right) ,x_{i}+\delta \left(
x_{i}\right) \right) $. By \textit{Cousin's} lemma the set of $\delta $%
\textit{-fine} partitions of $\left[ a,b\right] $ is nonempty, \cite{4}.

The collection $\mathcal{I}\left( \left[ a,b\right] \right) $ is the family
of compact subintervals $I$ of $\left[ a,b\right] $. The \textit{Lebesgue}
measure of the interval $I$ is denoted by $\left| I\right| $. Any real
valued function defined on $\mathcal{I}\left( \left[ a,b\right] \right) $ is
an interval function. For a function $f:\left[ a,b\right] \mapsto \mathbb{R}$%
, the associated interval function of $f$ is an interval function $f:%
\mathcal{I}\left( \left[ a,b\right] \right) \mapsto \mathbb{R}$, again
denoted by $f$, \cite{5}. If $f\equiv 0$ on $\left[ a,b\right] $ then its
associated interval function is trivial.

A function $f:\left[ a,b\right] \mapsto \mathbb{R}$ is said to be \textit{%
Kurzweil-Henstock} integrable to a real number $A$ on $\left[ a,b\right] $
if for every $\varepsilon >0$ there exists a gauge $\delta _{\varepsilon }:%
\left[ a,b\right] \mapsto \mathbb{R}_{+}$ such that $\left| \sum_{i\leq \nu
}[f\left( x_{i}\right) \left| \left[ a_{i},b_{i}\right] \right| ]-A\right|
<\varepsilon $, whenever $P\left[ a,b\right] $ is a $\delta _{\varepsilon }$-%
\textit{fine} partition of $\left[ a,b\right] $. In symbols, $A=\mathcal{KH-}%
\int_{a}^{b}f$.

\section{Main results}

In what follows we will use the following notations 
\begin{equation}
\Xi _{f}\left( P\left[ a,b\right] \right) =\sum_{i\leq \nu }[f\left(
x_{i}\right) \left| b_{i}-a_{i}\right| ]\text{ and }\Sigma _{\Phi }\left( P%
\left[ a,b\right] \right) =\sum_{i\leq \nu }[\Phi \left( b_{i}\right) -\Phi
\left( a_{i}\right) ]\text{.}  \label{2}
\end{equation}
Now, we are in a position to introduce the total \textit{Kurzweil-Henstock}
integral.

\begin{definition}
For any compact interval $\left[ a,b\right] \in \mathbb{R}$ let $E$ be a
non-empty subset of $\left[ a,b\right] $. A function $f:\left[ a,b\right]
\mapsto \mathbb{R}$ is said to be totally \textit{Kurzweil-Henstock }%
integrable to a real number $\Im $ on $\left[ a,b\right] $ if there exists a
nontrivial interval function $\Phi :\mathcal{I}\left( \left[ a,b\right]
\right) \mapsto \mathbb{R}$ with the following property: for every $%
\varepsilon >0$ there exists a gauge $\delta _{\varepsilon }$ on $\left[ a,b%
\right] $ such that $\left\vert \Xi _{f}\left( P\left[ a,b\right] \right)
-\Sigma _{\Phi }\left( P\left[ a,b\right] \left\vert _{\left[ a,b\right]
\backslash E}\right. \right) \right\vert <\varepsilon $ and $\Sigma _{\Phi
}\left( P\left[ a,b\right] \right) =\Im $, whenever $P\left[ a,b\right] \in 
\mathcal{P}\left[ a,b\right] $ is a $\delta _{\varepsilon }$\textit{-fine}
partition and\textit{\ }$P\left[ a,b\right] \left\vert _{\left[ a,b\right]
\backslash E}\right. $ is its restriction to $\left[ a,b\right] \backslash E$%
. In symbols, $\mathcal{KH-}vt\int_{a}^{b}f=\Im $.
\end{definition}

\begin{definition}
Let $E$ be a non-empty subset of $\left[ a,b\right] $. Then, an interval
function $\Phi :\mathcal{I}\left( \left[ a,b\right] \right) \mapsto \mathbb{R%
}$ is said to be basically summable $($BS$_{\delta _{\varepsilon }})$ to the
sum $\mathbf{\Re }$ on $E$ if there exists a real number $\mathbf{\Re }$
with the following property: given $\varepsilon >0$ there exists a gauge $%
\delta _{\varepsilon }$ on $\left[ a,b\right] $ such that $\left\vert \Sigma
_{\Phi }\left( P\left[ a,b\right] \left\vert _{E}\right. \right) -\Re
\right\vert <\varepsilon $, whenever $P\left[ a,b\right] \in \mathcal{P}%
\left[ a,b\right] $ is a $\delta _{\varepsilon }$\textit{-fine }partition and%
\textit{\ }$P\left[ a,b\right] \left\vert _{E}\right. $ is its restriction
to $E$. If $E$ can be written as a countable union of sets on each of which
the interval function $\Phi $ is BS$_{\delta _{\varepsilon }}$, then $\Phi $
is said to be BSG$_{\delta _{\varepsilon }}$ on $E$.
\end{definition}

Our main result reads as follows.

\begin{theorem}
For any compact interval $\left[ a,b\right] \in \mathbb{R}$ let $E$ be a
non-empty subset of $\left[ a,b\right] $ at whose points a real valued
function $F$ can take values $\pm \infty $ or not be defined at all. If $F$
is defined and differentiable on the set $\left[ a,b\right] \backslash E$,
then $D_{ex}F$ is totally \textit{Kurzweil-Henstock} integrable on $\left[
a,b\right] $ and 
\begin{equation}
\mathcal{KH-}vt\int_{a}^{b}D_{ex}F=F\left( b\right) -F\left( a\right) \text{.%
}  \label{3}
\end{equation}%
If the associated interval function of $F_{ex}$ defined by (\ref{1}) is in
addition basically summable $($BS$_{\delta _{\varepsilon }})$ to the sum $%
\mathbf{\Re }$ on $E$, then 
\begin{equation}
F\left( b\right) -F\left( a\right) =\mathcal{KH-}\int_{a}^{b}D_{ex}F+\mathbf{%
\Re }.  \label{4}
\end{equation}
\end{theorem}

Before starting with the proof we give the following lemma.

\begin{lemma}
Let $E$ be a non-empty subset of $\left[ a,b\right] $. If a function $f:%
\left[ a,b\right] \mapsto \mathbb{R}$ is totally \textit{Kurzweil-Henstock }%
integrable on $\left[ a,b\right] $ and $\Phi $ is basically summable $($BS$%
_{\delta _{\varepsilon }})$ to the sum $\mathbf{\Re }$ on $E$, then $f$ is 
\textit{Kurzweil-Henstock }integrable on $\left[ a,b\right] $ and 
\begin{equation}
\mathcal{KH-}vt\int_{a}^{b}f=\mathcal{KH-}\int_{a}^{b}f+\mathbf{\Re }\text{.}
\label{5}
\end{equation}
\end{lemma}

\begin{proof}
Given $\varepsilon >0$ we will construct a gauge for $f$ as follows. Since $f
$ is totally \textit{Kurzweil-Henstock }integrable on $\left[ a,b\right] $
it follows from \textit{Definition 1 }that there exist a real number $\Im $
and an interval function $\Phi $ with the following property: for every $%
\varepsilon >0$ there exists a gauge $\delta _{\varepsilon }^{\ast }$ on $%
\left[ a,b\right] $ such that $\left\vert \Xi _{f}\left( P\left[ a,b\right]
\right) -\Sigma _{\Phi }\left( P\left[ a,b\right] \left\vert _{\left[ a,b%
\right] \backslash E}\right. \right) \right\vert <\varepsilon $ and $\Sigma
_{\Phi }\left( P\left[ a,b\right] \right) =\Im $, whenever $P\left[ a,b%
\right] \in \mathcal{P}\left[ a,b\right] $ is a $\delta _{\varepsilon
}^{\ast }$\textit{-fine} partition and\textit{\ }$P\left[ a,b\right]
\left\vert _{\left[ a,b\right] \backslash E}\right. $ is its restriction to $%
\left[ a,b\right] \backslash E$. Choose a gauge $\delta _{\varepsilon
}^{\star }\left( x\right) $ as required in \textit{Definition 2 }above. The
function $\delta _{\varepsilon }=min\left( \delta _{\varepsilon }^{\ast
},\delta _{\varepsilon }^{\star }\right) $ is a gauge on $\left[ a,b\right] $%
. We now let $P\left[ a,b\right] =\left\{ \left( \left[ a_{i},b_{i}\right]
,x_{i}\right) \mid i=1,...,\nu \right\} $ be a $\delta _{\varepsilon }$%
\textit{-fine }partition of $\left[ a,b\right] $. It is readily seen that 
\begin{equation*}
\left\vert \Xi _{f}\left( P\left[ a,b\right] \right) -\Im +\mathbf{\Re }%
\right\vert =
\end{equation*}%
\begin{equation*}
=\left\vert \Xi _{f}\left( P\left[ a,b\right] \right) -\Im +\Sigma _{\Phi
}\left( P\left[ a,b\right] \left\vert _{E}\right. \right) -[\Sigma _{\Phi
}\left( P\left[ a,b\right] \left\vert _{E}\right. \right) -\mathbf{\Re ]}%
\right\vert \leq 
\end{equation*}%
\begin{equation*}
\leq \left\vert \Xi _{f}\left( P\left[ a,b\right] \right) -\Sigma _{\Phi
}\left( P\left[ a,b\right] \left\vert _{\left[ a,b\right] \backslash
E}\right. \right) \right\vert +\left\vert \Sigma _{\Phi }\left( P\left[ a,b%
\right] \left\vert _{E}\right. \right) -\Re \right\vert <2\varepsilon \text{.%
}
\end{equation*}%
Therefore, $f$ is \textit{Kurzweil-Henstock} integrable on $\left[ a,b\right]
$ and $\mathcal{KH-}\int_{a}^{b}f=\Im -\mathbf{\Re }$, that is 
\begin{equation*}
\mathcal{KH-}vt\int_{a}^{b}f=\mathcal{KH-}\int_{a}^{b}f+\mathbf{\Re }\text{.}
\end{equation*}
\end{proof}

We now turn to the proof of \textit{Theorem 1}.

\begin{proof}
Given $\varepsilon >0$. By definition of $f$ at the point $x\in \left[ a,b%
\right] \backslash E$, given $\varepsilon >0$ there exists $\delta
_{\varepsilon }\left( x\right) >0$ such that if $x\in \left[ u,v\right]
\subseteq \left[ x-\delta _{\varepsilon }\left( x\right) ,x+\delta
_{\varepsilon }\left( x\right) \right] $ and $x\in \left[ a,b\right]
\backslash E$, then 
\begin{equation*}
\left\vert F\left( v\right) -F\left( u\right) -f\left( x\right) \left(
v-u\right) \right\vert <\varepsilon \left( v-u\right) \text{.}
\end{equation*}%
For $F_{ex}$ defined by (\ref{1}) let $F_{ex}:\mathcal{I}\left( \left[ a,b%
\right] \right) \mapsto \mathbb{R}$ be its associated interval function. We
now let $P\left[ a,b\right] =\left\{ \left( \left[ a_{i},b_{i}\right]
,x_{i}\right) \mid i=1,...,\nu \right\} $ be a $\delta _{\varepsilon }$%
\textit{-fine }partition of $\left[ a,b\right] $. Since $F\left( b\right)
-F\left( a\right) =\sum_{i=1}^{\nu }\left[ F_{ex}\left( b_{i}\right)
-F_{ex}\left( a_{i}\right) \right] $ and (remember if $x\in E$, then $%
D_{ex}F=0$)%
\begin{equation*}
\left\vert \Xi _{f}\left( P\left[ a,b\right] \right) -\Sigma _{F}\left( P%
\left[ a,b\right] \left\vert _{\left[ a,b\right] \backslash E}\right.
\right) \right\vert =
\end{equation*}%
\begin{equation*}
=\left\vert \Xi _{f}\left( P\left[ a,b\right] \left\vert _{\left[ a,b\right]
\backslash E}\right. \right) -\Sigma _{F}\left( P\left[ a,b\right]
\left\vert _{\left[ a,b\right] \backslash E}\right. \right) \right\vert
<\varepsilon \left( b-a\right) \text{,}
\end{equation*}%
it follows from \textit{Definition 1} that $D_{ex}F$ is totally \textit{%
Kurzweil-Henstock} integrable on $\left[ a,b\right] $ and 
\begin{equation*}
\mathcal{KH-}vt\int_{a}^{b}D_{ex}F=F\left( b\right) -F\left( a\right) \text{.%
}
\end{equation*}%
Finally, based on the result of \textit{Lemma 1} 
\begin{equation*}
F\left( b\right) -F\left( a\right) =\mathcal{KH-}\int_{a}^{b}D_{ex}F+\mathbf{%
\Re }\text{.}
\end{equation*}
\end{proof}

By \textit{Definition 2 }one can easily see that if $\mathbf{\Re }=0$ then $%
F $ has negligible variation on $E$, \cite[Definition 5.11]{1}. So, we now
in position to define a residual function of $F$.

\begin{definition}
Let $F:\left[ a,b\right] \mapsto \mathbb{R}$. A function $\mathcal{R}:\left[
a,b\right] \mapsto \mathbb{R}$ is said to be a residual function of $F$ on $%
\left[ a,b\right] $ if given $\varepsilon >0$ there exists a gauge $\delta
_{\varepsilon }$ on $\left[ a,b\right] $ such that $\left| F\left(
b_{i}\right) -F\left( a_{i}\right) -\mathcal{R}\left( x_{i}\right) \right|
<\varepsilon $, whenever $P\left[ a,b\right] \in \mathcal{P}\left[ a,b\right]
$ is a $\delta _{\varepsilon }$\textit{-fine }partition.
\end{definition}

\begin{definition}
Let $E$ be a non-empty subset of $\left[ a,b\right] $ and let $F:\left[ a,b%
\right] \mapsto \mathbb{R}$ be a function whose associated interval function 
$F:\mathcal{I}\left( \left[ a,b\right] \right) \mapsto \mathbb{R}$ is BS$%
_{\delta _{\varepsilon }}$ $($BSG$_{\delta _{\varepsilon }})$ to the sum $%
\mathbf{\Re }$ on $E$. Then, a residual function $\mathcal{R}:\left[ a,b%
\right] \mapsto \mathbb{R}$ of $F$ is said to be also BS$_{\delta
_{\varepsilon }}$ $($BSG$_{\delta _{\varepsilon }})$ to the same sum $%
\mathbf{\Re }$ on $E$. In symbols, $\sum_{x\in E}\mathcal{R}\left( x\right) =%
\mathbf{\Re }$.
\end{definition}

Clearly, \textit{Definition 4 }establishes a causal connection between%
\textit{\ Definitions 2 }and\textit{\ 3}. If $E$ is a countable set, the
causality is so obvious. However, if $E$ is an infinite set, then this
connection is not necessarily a causal connection. Namely, if $F:\left[ a,b%
\right] \mapsto \mathbb{R}$ has negligible variation on some subset $E$ of $%
\left[ a,b\right] $, which is a countably infinite set, then its residual
function $\mathcal{R}$ vanishes identically on $E$, so that the sum $%
\sum_{x\in E}\mathcal{R}\left( x\right) $ is reduced to the so-called
indeterminate expression $\infty \cdot 0$ that have, in this case, the null
value. On the contrary, if $F$ has no negligible variation on $E$, and its
residual function $\mathcal{R}$ also vanishes identically on $E$, as in the
case of the \textit{Cantor} function, then the sum $\sum_{x\in E}\mathcal{R}%
\left( x\right) $ is reduced to the indeterminate expression $\infty \cdot 0$
that actually have, in \textit{Cantor's} case, the numerical value of $1$.
By \textit{Definition 4}, we may rewrite (\ref{4}) as follows, 
\begin{equation}
F\left( b\right) -F\left( a\right) =\mathcal{KH-}\int_{a}^{b}D_{ex}F+\sum_{x%
\in E}\mathcal{R}\left( x\right) \text{.}  \label{6}
\end{equation}
If $f$ in addition vanishes identically on $\left[ a,b\right] \backslash E$,
then 
\begin{equation}
F\left( b\right) -F\left( a\right) =\sum_{x\in E}\mathcal{R}\left( x\right) .
\label{7}
\end{equation}
The previous result is an extension of \textit{Cauchy's} residue theorem
result in $\mathbb{R}$.

\section{Examples}

For an illustration of (\ref{6}) and (\ref{7}) we consider the \textit{%
Heaviside} unit function defined by 
\begin{equation}
F\left( x\right) =\left\{ 
\begin{array}{c}
0\text{, if }a\leq x\leq 0 \\ 
1\text{, if }0<x\leq b%
\end{array}
\right. \text{.}  \label{8}
\end{equation}
In this case, if $a<0$, then $\mathcal{KH-}vt\int_{a}^{b}D_{ex}F=1$, in
spite of the fact that $D_{ex}F\equiv 0$ on $\left[ a,b\right] $.
Accordingly, it follows from (\ref{6}) and (\ref{7}) that $\mathcal{R}\left(
0\right) =1$, since 
\begin{equation}
f\left( x\right) =\left\{ 
\begin{array}{c}
+\infty \text{, if }x=0 \\ 
0\text{, otherwise}%
\end{array}
\right. \text{,}  \label{9}
\end{equation}
where $f$ is the derivative of $F$, and $\mathcal{KH-}\int_{a}^{b}D_{ex}F=0$.

Let $\left[ a,b\right] \subset \mathbb{R}$ be an arbitrary compact interval
within which is the point $x=0$. For an illustration of the result (\ref{3})
of \textit{Theorem 1} we consider the real valued function $F\left( x\right)
=1/x$ that is differentiable to $f\left( x\right) =-\left( 1/x^{2}\right) $
at all but the exceptional set $\left\{ 0\right\} $ of $\left[ a,b\right] $.
In spite of the fact that $f$ is not \textit{Kurzweil-Henstock }integrable
on $\left[ a,b\right] $ it follows from (\ref{3}) that $\mathcal{KH-}%
vt\int_{a}^{b}D_{ex}F=\left( a-b\right) /\left( ab\right) $. In this case, $%
\mathcal{R}\left( x\right) $ is not defined at the point $x=0$, that is 
\begin{equation}
\mathcal{R}\left( x\right) =\left\{ 
\begin{array}{c}
+\infty \text{, if }x=0 \\ 
0\text{, otherwise}%
\end{array}
\right. \text{,}  \label{10}
\end{equation}
and $\mathcal{KH-}vt\int_{a}^{b}D_{ex}F$ is reduced to the so-called
indeterminate expression $\infty -\infty $ (in the sense of the difference
of limits) that actually have, in this situation, the real numerical value
of $\left( a-b\right) /\left( ab\right) $.

\end{document}